\documentclass[10pt]{article}

\usepackage{a4wide}
\usepackage[english]{babel}
\usepackage{amssymb}
\usepackage{amsmath}
\usepackage{amsthm}
\usepackage{array}
\usepackage{graphicx}
\usepackage{verbatim}
\usepackage{url}
\usepackage{float}
\usepackage{framed}
\usepackage{ifthen}
\usepackage{dcolumn} \newcolumntype{.}{D{.}{.}{3.1}}
\usepackage[usenames,dvipsnames]{color}

\renewcommand{\qed}{\hfill\ensuremath{\Box}}
\renewcommand{\phi}{\varphi}
\renewcommand{\theta}{\vartheta}
\renewcommand{\epsilon}{\varepsilon}

\newtheorem{theorem}[equation]{Theorem}
\newtheorem{lemma}[equation]{Lemma}
\newtheorem{proposition}[equation]{Proposition}

\theoremstyle{definition}

\renewenvironment{proof}[1][Proof]{\begin{trivlist}\item[\hskip \labelsep {\bfseries #1}]}{\qed\end{trivlist}}

\makeatletter
\newenvironment{myleftbar}{%
  \MakeFramed {\advance\hsize-\width}
}{%
  \endMakeFramed%
}
\makeatother
\newtheoremstyle{example}
  {\topsep} {\topsep}%
  {\upshape}
  {}
  {\bfseries}
  {.}
  {\parindent}
  {\thmname{#1}\thmnumber{ #2}\thmnote{#3}}
\theoremstyle{example}
\newtheorem{xexample}[equation]{Example}
\newenvironment{example}{ \begin{myleftbar} \begin{xexample} }{ \end{xexample}\end{myleftbar} }

\newtheoremstyle{examplecont}
  {\topsep} {\topsep}%
  {\upshape}
  {}
  {\bfseries}
  {.}
  {\parindent}
  {\thmname{#1}\thmnumber{ #2} \thmnote{#3}\enspace(continued)}
\theoremstyle{examplecont}
\newtheorem*{xexamplecont}{Example}

\newfloat{myalg}{tbp}{alg} \floatname{myalg}{Algorithm}
\newenvironment{alg}{\setcounter{myalg}{\arabic{equation}} \begin{myalg}}{\end{myalg}\setcounter{equation}{\arabic{myalg}}}

\newcommand{\tmpalgoname}{}
\newcommand{\tmpalgocap}{}
\newcommand{\tmpalgolbl}{}
\newenvironment{algorithm}[4][def]{%
  \ifthenelse{\equal{#1}{def}}{\renewcommand{\tmpalgoidx}{#2}}{\renewcommand{\tmpalgoidx}{#1}}%
  \renewcommand{\tmpalgoname}{#2}%
  \renewcommand{\tmpalgocap}{#3}%
  \renewcommand{\tmpalgolbl}{#4}%
  \begin{alg}%
  \begin{tabbing}%
  \quad\quad\=\quad\=\quad\=\quad\=\quad\=\quad\=\quad\=\kill%
  {\sc \tmpalgoname} \\
}{%
  \end{tabbing}%
  \caption{\tmpalgocap}\label{alg_\tmpalgolbl}%
  \end{alg} %
}
\newcounter{algln}
\newcommand{\lnreset}{\setcounter{algln}{0}}
\newcommand{\lnp}{\addtocounter{algln}{1} {\footnotesize \arabic{algln}}}

\newcommand{\compdata}[8]
{	
	\begin{minipage}{5mm} 
		(#1) \\
	\end{minipage} 
	&
	\begin{minipage}{20mm}
		\begin{tabular}{c}
			\includegraphics[height=18mm]{#2} \cr
			\vspace{1mm}
		\end{tabular}
	\end{minipage} 
	&
	\begin{minipage}{45mm}
		\begin{tabular}{ll}
			$\dim(X)$:  & $#3$ \\
			$\dim(L)$:  & $#4$ \\
			$L/\Rad(L)$: & \ifthenelse{\equal{#5}{0}}{#6}{$#5$-dim \ifthenelse{\equal{#6}{}}{}{($#6$)}} \\
			runtime:    & #7 \\
			\ifthenelse{\equal{#8}{}}{}{\multicolumn{2}{l}{\emph{#8}} \\}
		\end{tabular}
	\end{minipage} 
}

\renewcommand{\theenumi}{(\roman{enumi})}

\newcommand{\Magma}{{\sc Magma}}

\newcommand{\ad}{\operatorname{ad}}
\newcommand{\bigO}{O}
\newcommand{\Norm}{\operatorname{N}}
\newcommand{\F}{\mathbb{F}}
\newcommand{\SC}{^\mathrm{SC}}
\newcommand{\Ad}{^\mathrm{ad}}
\newcommand{\GF}{\operatorname{GF}}
\newcommand{\ch}{^\vee}
\newcommand{\chr}{\operatorname{char}}
\newcommand{\Center}{\operatorname{Z}}
\newcommand{\Cent}{\operatorname{C}}
\newcommand{\Hom}{\operatorname{Hom}}
\newcommand{\Lie}{\operatorname{Lie}}
\newcommand{\LZ}{L_{\mathbb Z}}
\newcommand{\LF}{L_{\mathbb F}}
\newcommand{\rk}{\operatorname{rk}}
\newcommand{\Z}{\mathbb{Z}}

\title{Computing Split Maximal Toral Subalgebras of Lie Algebras over Fields of Small Characteristic}
\author{Dan Roozemond, University of Sydney}

\begin{document}
\maketitle

\begin{abstract}
Important subalgebras of a Lie algebra of an algebraic group are its toral subalgebras, or equivalently (over fields of characteristic 0) its Cartan subalgebras. Of great importance among these are ones that are split: their action on the Lie algebra splits completely over the field of definition.
While algorithms to compute split maximal toral subalgebras exist and have been implemented \cite{Ryba07,CM06}, these algorithms fail when the Lie algebra is defined over a field of characteristic $2$ or $3$. 

We present heuristic algorithms that, given a reductive Lie algebra $L$ over a finite field of characteristic $2$ or $3$, find a split maximal toral subalgebra of $L$. Together with earlier work \cite{CR09} these algorithms are very useful for the recognition of reductive Lie algebras over such fields.
\end{abstract}

\section{Introduction}\label{sec_introduction}
For computational problems regarding a split reductive algebraic group $G$
defined over a field $\F$, it is often useful to calculate within its Lie
algebra $L$ over $\F$. For instance, the conjugacy question for two split
maximal tori in $G$ can often be translated to a conjugacy question for two
split toral subalgebras of $L$. A \emph{maximal toral subalgebra} $H$ of
$L$ is a commutative subalgebra consisting of semisimple elements (recall 
that an element $x \in L$ is called semisimple if it is contained in the 
$p$-subalgebra generated by $x^p$), and it is maximal (with respect to 
inclusion) among subalgebras of $L$ with these properties. 
For such a subalgebra we have that multiplication (in $L$) by each of 
its elements is semisimple, i.e.,~each of its elements has a
diagonal form with respect to a suitable basis over a large enough extension
field of $\F$. A maximal toral subalgebra $H$ is called \emph{split} (or
$\F$-\emph{split}) if, for every $h\in H$, left multiplication by $h$, 
denoted $\ad_h$, has a diagonal form with respect to a suitable basis over $\F$.
Such a subalgebra is the Lie algebra of a split maximal torus in $G$.
Recall that $\hat{H}$ is called a \emph{Cartan subalgebra} of $L$ if $\hat{H}$
is nilpotent and $\Norm_L(\hat{H}) = \hat{H}$.
Maximal toral subalgebras and Cartan subalgebras are closely related:
if $H$ is a maximal toral subalgebra of $L$, then $\Cent_L(H)$ is a
Cartan subalgebra of $L$ \cite[Proposition 15.1]{Hum67}.

In the case that $\F$ is not of characteristic $2$ or $3$ a Las Vegas algorithm exists to compute split maximal toral subalgebras, due to Cohen and Murray \cite[Lemma 5.7]{CM06}. Independently, Ryba developed a Las Vegas algorithm for computing split Cartan subalgebras \cite{Ryba07}. Unfortunately, Ryba also excludes characteristic $2$ and, if the Lie algebra is of type $\mathrm A_2$ or $\mathrm G_2$, characteristic $3$. It is, however, claimed that the algorithm may work in some cases in characteristic $2$, but not in all cases (cf.~\cite[Section 9.3]{Ryba07}).

In this paper we present heuristic algorithms that, given a reductive Lie algebra $L$ over a finite field of characteristic $2$ or $3$, find a split maximal toral subalgebra of $L$. We present separate algorithms for the two characteristics. Each of these algorithms is Las Vegas: it either returns a subalgebra $H$ of the correct form (with nonzero probability), or it returns \textbf{fail}. 
The algorithm for the characteristic $2$ case has been described before in the author's PhD thesis \cite[Chapter 3]{roozemond10}.

\subsection*{The remainder of this paper}
The remainder of this section sets the theoretical framework we need to describe the algorithms.
In Section \ref{sec_scsa_C4} we investigate a particular Lie algebra and a split maximal toral subalgebra that is not contained in a split toral subalgebra of maximal dimension. In Section \ref{sec_stsa_countingrss_new} we study the occurrence (or lack thereof) of regular semisimple elements in Lie algebras over fields of characteristic $2$, showing that the Las Vegas algorithm by Cohen and Murray cannot easily be applied in those cases. 
In Section \ref{sec_mstsa_algorithm3} we describe a heuristic algorithm to find split maximal toral subalgebras in the characteristic $3$ case, and 
in Section \ref{sec_mstsa_algorithm2} we describe such an algorithm for the characteristic $2$ case. 
Finally, in Section \ref{sec_stsa_notesimpl} we comment on the implementation and performance of these algorithms.

\subsection*{Root data}
Our treatment of Lie algebras and the corresponding algebraic groups rests on
the theory developed mainly by Chevalley and available in the excellent books
by Borel \cite{Bor91}, Humphreys \cite{Hum75}, and Springer \cite{Springer98}.
Our set-up is as in \cite{CR09}, a publication we will repeatedly refer to 
because of the extensive analysis it contains on the structure of reductive Lie 
algebras over fields of small characteristic.
We refer to \cite{CR09} for more details on our set-up and restrict ourselves 
to the essential notions here.

Split reductive algebraic groups are determined by their fields of definition
and their root data \cite[Theorem 9.4.3]{Springer98}. Throughout this paper we let $R
= (X, \Phi, Y, \Phi\ch)$ be a \emph{root datum of rank} $n$. 
This means $X$ and $Y$ are dual free $\Z$-modules of dimension
$n$ with a bilinear pairing $\langle \cdot, \cdot \rangle : X \times Y
\rightarrow \Z$ such that the induced map $X \rightarrow \Hom(Y, \Z)$
is an isomorphism (and then so is the induced map $Y \rightarrow \Hom(X, \Z)$), 
$\Phi$ is a finite subset of $X$ and $\Phi\ch$ a finite subset of $Y$, 
and called the \emph{roots} and \emph{coroots},
respectively, and there is a one-to-one correspondence $\ch : \Phi \rightarrow
\Phi\ch$ such that $\langle \alpha, \alpha\ch \rangle = 2$ for all $\alpha \in
\Phi$. Both the roots $\Phi$ and the coroots $\Phi\ch$ should form a root system 
in the traditional sense.
The irreducible root systems are well-known, and described in Cartan's notation 
${\rm A}_n$ $(n\ge1)$, ${\rm B}_n$ $(n\ge2)$, ${\rm C}_n$ $(n\ge3)$, 
${\rm D}_n$ $(n\ge4)$, ${\rm E}_n$ $(n\in\{6,7,8\})$, ${\rm F}_4$, ${\rm G}_2$. 

A \emph{weight} is a vector $w$ in the Euclidian space $X \otimes \mathbb{R}$
such that $\langle w, \alpha\ch \rangle \in \Z$ for all $\alpha \in \Phi$.
These weights form a weight lattice, and the \emph{fundamental group} is defined
to be the quotient of this weight lattice by the root lattice $\Z \Phi$.
The subgroups of this fundamental group parametrize the possible semisimple root data with a given
root system $\Phi$ via the quotient $X/\Z \Phi$. 
For sake of completeness we remark that the fundamental group is 
$\Z/(n+1)\Z$ for $\mathrm A_n$; $\Z/2\Z$ for $\mathrm B_n$ and $\mathrm C_n$;
$\Z/4\Z$ for $\mathrm D_n$ if $n$ is odd, $\Z/2\Z+\Z/2\Z$ for $\mathrm D_n$ 
if $n$ is even; $\Z/3\Z$ for $\mathrm E_6$; $\Z/2\Z$ for $\mathrm E_7$; 
and it is trivial for $\mathrm E_8$, $\mathrm F_4$, and $\mathrm G_2$.

We use this observation to 
define the \emph{isogeny type} of a root datum. If $X/\Z \Phi$ is the trivial
group, $R$ is said to be of \emph{adjoint} isogeny type; if on the other hand
$X/\Z \Phi$ is the full fundamental group, $R$ is said to be of \emph{simply 
connected} isogeny type. If neither of these is the case, $R$ is said to be
of \emph{intermediate} isogeny type. Note that the last case only occurs for
root systems of type $\mathrm A_n$ and $\mathrm D_n$.

We denote an adjoint root datum whose root system is of type $\mathrm X_n$ by
$\mathrm X_n^{\mathrm{ad}}$, and its simply connected variant by $\mathrm X_n^{\mathrm{SC}}$.
Intermediate root data of type $\mathrm A_n$ will be denoted by $\mathrm A_n^{(k)}$,
where $k | n+1$; intermediate root data of type $\mathrm D_n$ will be denoted by
$\mathrm D_n^{(1)}$ if $n$ is odd, and by $\mathrm D_n^{(1)}$, $\mathrm D_n^{(n-1)}$,
and $\mathrm D_n^{(n)}$ if $n$ is even.

\begin{example}\label{ex_rootdata_A1}
There are two root data of type $\mathrm A_1$, namely $\mathrm A_1\Ad$ and $\mathrm A_1\SC$.
In both cases, $X = Y = \mathbb Z$; we may take $e_1 = (1)$ to be a basis of $X$ and $f_1 = (1)$
a basis of $Y$. 
For $\mathrm A_1\Ad$, the roots are taken to be $\alpha = (1)$, $-\alpha = (-1)$,
the coroots $\alpha\ch = (2)$, $-\alpha\ch = (-2)$, and the pairing between $X$ and $Y$ is
simply multiplication, so that $\langle \alpha, f_1 \rangle = 1$ and $\langle e_1, \alpha\ch \rangle = 2$. 

Conversely, for $\mathrm A_1\SC$, the roots are taken to be $\alpha = (2)$, $-\alpha = (-2)$,
the coroots $\alpha\ch = (1)$, $-\alpha\ch = (-1)$, and the pairing again is multiplication,
so that $\langle \alpha, f_1 \rangle = 2$ and $\langle e_1, \alpha\ch \rangle = 1$. 
\end{example}

\subsection*{Chevalley Lie algebras}
Given a root datum $R$ we consider the free $\Z$-module 
\[
\LZ(R) = Y \oplus \bigoplus_{\alpha \in \Phi} \Z X_\alpha,
\]
where the $X_\alpha$ are formal basis elements. The rank of $\LZ(R)$ is $n +
|\Phi|$.  We denote by $[\cdot,\cdot]$ the alternating bilinear map $\LZ(R)
\times \LZ(R) \rightarrow \LZ(R)$ determined by the following rules:
\[
\begin{array}{lrclr}
\mbox{For } y, z \in Y : & [y, z] & = & 0, & \mathrm{(CB}\mathbb Z\mathrm{1)}\\
\mbox{For } y \in Y, \alpha \in \Phi : & [ X_\alpha, y ] & = & \langle \alpha, y \rangle X_\alpha , & \mathrm{(CB}\mathbb Z\mathrm{2)}\\
\mbox{For } \alpha \in \Phi : & [ X_{-\alpha}, X_\alpha ] & = & \alpha\ch , & \mathrm{(CB}\mathbb Z\mathrm{3)}\\
\mbox{For } \alpha,\beta \in \Phi, \alpha \neq \pm \beta: & [X_\alpha, X_\beta] & = & 
  \left\{ \begin{array}{ll}
  N_{\alpha,\beta} X_{\alpha+\beta} & \mbox{if } \alpha + \beta \in \Phi, \cr
  0 & \mbox{otherwise.}
  \end{array} \right. & \mathrm{(CB}\mathbb Z\mathrm{4)}
\end{array}
\]
The $N_{\alpha,\beta}$ are integral structure constants chosen to be $\pm(p_{\alpha,\beta}+1)$, 
where $p_{\alpha,\beta}$ is the biggest number such that $\alpha
-p_{\alpha,\beta}\beta$ is a root and the signs are chosen (once and for all)
so as to satisfy the
Jacobi identity.
It is easily verified that $N_{\alpha,\beta} = -N_{-\alpha,-\beta}$ and it is a well-known result 
(see for example \cite[Proposition 4.4.2]{Car72}) that  such a product exists.
$\LZ(R)$ is called a \emph{Chevalley Lie algebra}.

For any field $\F$ we define $L_\F(R)$ to be the Lie algebra $\LZ(R) \otimes \F$.

\begin{example}\label{ex_chevlie_A1}
Corresponding to the two root data of type $\mathrm A_1$ described in Example \ref{ex_rootdata_A1}
there exist two Chevalley Lie algebras of type $\mathrm A_1$. Both are $3$-dimensional and have formal
basis elements $X_\alpha$, $X_{-\alpha}$, and $h$, but their multiplication differs.

The multiplication for the Lie algebra of type $\mathrm A_1\Ad$ is determined by $[X_\alpha, X_{-\alpha}] = -2h$, $[X_\alpha,h] = X_\alpha$, and $[X_{-\alpha},h] = -X_{-\alpha}$ (observe that the multiplication on all other algebra elements follows from bilinearity and anti-symmetry). 
On the other hand, for $\mathrm A_1\SC$ we have $[X_\alpha, X_{-\alpha}] = -h$, $[X_\alpha,h] = 2X_\alpha$, and $[X_{-\alpha},h] = -2X_{-\alpha}$. 

It is easy to see that $L_\F(\mathrm A_1\Ad)$ and $L_\F(\mathrm A_1\SC)$ are isomorphic unless
$\chr(\F) = 2$.
\end{example}

A basis of $\LZ(R)$ that consists of a basis of $Y$ and the formal elements
$X_\alpha$ and satisfies $\mathrm{(CB}\mathbb
Z\mathrm{1)}$--$\mathrm{(CB}\mathbb Z\mathrm{4)}$ is called a \emph{Chevalley
basis} of the Lie algebra $\LZ(R)$ with respect to the split Cartan subalgebra
$Y$ and the root datum $R$.  If no confusion is imminent we just call this a
\emph{Chevalley basis} of $\LZ(R)$.

The interest in Chevalley Lie algebras comes from the following result.

\begin{theorem}[Chevalley \cite{Chev58}]\label{thm_liealg_has_chevbas}
Suppose that $L$ is the Lie algebra of a split semisimple algebraic group $G$
over $\F$ with root datum $R = (X,\Phi,Y,\Phi\ch)$. Then $L \cong \LF(R)$,
and if $G$ is simple then $R$ is irreducible.
\end{theorem}

In light of this theorem, we will view all Lie algebras as Chevalley Lie algebras in the remainder of this paper. Moreover, we will only consider Chevalley Lie algebras with an irreducible root datum, as the algorithms presented easily generalise to the case of an arbitrary root datum.

We assume Lie algebras to be given as a vector space together with a list of structure constants that define the products of vectors, i.e., we are given a field $\F$, a $d$-dimensional vector space $V$ over $\F$ with basis $b_1, \ldots, b_d$, and structure constants $c_{ijk}$ that are understood to mean
\[
[b_i, b_j] = \sum_{k=1}^d c_{ijk} b_k.
\]

\begin{example}
We consider the Lie algebra for $\mathrm A_1\Ad$ defined in Example \ref{ex_chevlie_A1}. If we take the basis elements to be $b_1 = X_{\alpha}$, $b_2 = X_{-\alpha}$, and $b_3 = h$, the only nonzero $c_{ijk}$ are: $c_{123} = -2$, $c_{131} = 1$, $c_{213} = 2$, $c_{232} = -1$, $c_{311} = -1$, and $c_{322} = 1$.
\end{example}

This is the standard way of representing a finite dimensional Lie algebra on a computer (see \cite[Section 1.5]{deGraaf00}).
In general, of course, the basis we are given is arbitrary and not of a particularly nice form such as a Chevalley basis. In fact, when a Chevalley basis is known a split maximal toral subalgebra is easily recovered: it simply is $Y$.

\subsection*{Root spaces}
Let $R=(X,\Phi,Y,\Phi\ch)$ be a root datum, fix a basis $\{ y_1, \ldots, y_n \}$ of $Y$, let $\F$ be a field, $L = L_\F(R)$ the Lie algebra of type $R$ over $\F$, and let $H = Y \otimes \F$ and $h_i = y_i \otimes 1_\F$. We call $H$ the \emph{standard split maximal toral subalgebra} of $L$. We define a \emph{root of $H$ on $L$} to be the function
\[
	\overline{\alpha}: H \rightarrow \mathbb Z, h \mapsto \sum_{i=1}^n \langle \alpha, y_i \rangle t_i, \mbox{ where } h = \sum_{i=1}^n y_i \otimes t_i = \sum_{i=1}^n t_i h_i,
\]
where $n$ is the rank of $R$. Note that $\overline{\alpha}$ actually maps into $\F$, but by construction the image actually consists of integers (cf. Equations $\mathrm{(CB}\mathbb Z\mathrm{1)}$ -- $\mathrm{(CB}\mathbb Z\mathrm{4)}$).
Furthermore, we define the \emph{root space corresponding to $\overline{\alpha}$} to be
\[
L_{\alpha} = \bigcap_{i=1}^n \operatorname{Ker}(\ad_{h_i} - \overline{\alpha}(h_i)).
\]

\begin{example}
We consider Lie algebras with root data $\mathrm A_2\Ad$ and $\mathrm A_2\SC$ over a number of different fields. We denote the roots by $\pm \alpha_1, \pm \alpha_2, \pm(\alpha_1+\alpha_2)$.

First, suppose $\F = \mathbb Q$. In this case, for both $\mathrm A_2\Ad$ and $\mathrm A_2\SC$, all $L_\alpha$ are $1$-dimensional and distinct. For example, $L_{\alpha_2} = \mathbb Q X_{\alpha_2}$.

Second, suppose $\F$ is a field of characteristic $2$. Then, for both $\mathrm A_2\Ad$ and $\mathrm A_2\SC$, all $L_\alpha$ are $2$-dimensional. For example, $L_{\alpha_1} = L_{-\alpha_1} = \F X_{\alpha_1} + \F X_{-\alpha_1}$.

Finally, suppose $\F$ is a field of characteristic $3$. Recall from Equation (CB$\mathbb Z$2) that the action of $H$ on $L$ (in particular on the $X_\alpha$) depends on the isogeny type of the root datum, so that the root spaces may differ between $\mathrm A_2\Ad$ and $\mathrm A_2\SC$. In this case they indeed do. For $\mathrm A_2\Ad$, all $L_\alpha$ are $1$-dimensional and distinct. However, for $\mathrm A_2\SC$, we have 
\begin{align*}
L_{\alpha_1} = L_{\alpha_2} = L_{-\alpha_1-\alpha_2} & = \F X_{\alpha_1} + \F X_{\alpha_2} + \F X_{-\alpha_1-\alpha_2}, \mbox{ and} \cr
L_{-\alpha_1} = L_{-\alpha_2} = L_{\alpha_1+\alpha_2} & = \F X_{-\alpha_1} + \F X_{-\alpha_2} + \F X_{\alpha_1+\alpha_2}.
\end{align*}
\end{example}

\section{A characteristic $2$ curiosity}\label{sec_scsa_C4}
For the development of a recursive algorithm for finding split maximal toral subalgebras it would be very helpful 
to know in advance that every split toral subalgebra is contained in a split toral subalgebra that is maximal among all (not necessarily split) toral subalgebras.
The algorithm by Cohen and Murray relies on a similar (but weaker) assertion (cf.~\cite[Proposition 5.8]{CM06}).
This is, however, not in general true in characteristic $2$, as we will show in the following example.

We consider the Chevalley Lie algebra $L$ of type $\mathrm C_4\SC$
over $\GF(2)$, with root datum $R = (X,\Phi,Y,\Phi\ch)$ and Chevalley basis elements $\{ X_\alpha, h_i \mid \alpha \in \Phi, i \in \{1, \ldots, 4\} \}$.
Furthermore, we denote the simple roots of $\Phi$ by $\alpha_1, \ldots, \alpha_4$, where $\alpha_1, \alpha_2$, and $\alpha_3$ are short roots, and $\alpha_4$ is long. Its non-simple positive roots are then
\begin{align*}
	\alpha_{5} = (1, 1, 0, 0), \alpha_{6} = (0, 1, 1, 0), \alpha_{7} = (0, 0, 1, 1), \alpha_{8} = (1, 1, 1, 0), \cr
	\alpha_{9} = (0, 1, 1, 1),\alpha_{10} = (0, 0, 2, 1), \alpha_{11} = (1, 1, 1, 1),	\alpha_{12} = (0, 1, 2, 1), \cr
	\alpha_{13} = (1, 1, 2, 1), \alpha_{14} = (0, 2, 2, 1), \alpha_{15} = (1, 2, 2, 1),	\alpha_{16} = (2, 2, 2, 1),	
\end{align*}
where $(c_1, c_2, c_3, c_4)$ denotes $c_1 \alpha_1 + c_2 \alpha_2 + c_3 \alpha_3 + c_4 \alpha_4$, and the negative roots are defined accordingly.
Now let 
\begin{align*}
y_1 & = h_1 + h_3 \in \Center(L),\cr
y_2 & = h_1 + X_{\alpha_{12}} + X_{-\alpha_{8}}, \cr
y_3 & = h_2 + X_{\alpha_{3}} + X_{-\alpha_{3}} + X_{\alpha_{15}} + X_{-\alpha_{15}},
\end{align*}
and $H = \langle y_1, y_2, y_3 \rangle_L$.

\newcommand{\Hbar}{H'}
\newcommand{\Lzbar}{L_0'}
\begin{proposition}\label{prop_mtsa_counterex}
The subalgebra $H$ is a $3$-dimensional split toral subalgebra of $L$. 
There, however, does not exist a split toral subalgebra $\Hbar$ of $L$ of dimension $4$ such that $H \subseteq \Hbar$.
\end{proposition}
\begin{proof}
It is straightforward to verify that $H$ is a split toral subalgebra of $L$: on diagonalization of $H$ in the adjoint representation we obtain $3$ eigenspaces of dimension $8$ (corresponding to roots $(0,1,0)$, $(0,0,1)$, and $(0,1,1)$) and an eigenspace $L_0$ of dimension $12$ (corresponding to the root $(0,0,0)$ and $H$ itself). 

Now suppose there exists a split toral subalgebra $\Hbar$ of dimension $4$ containing $H$.
This would imply the existence of a $y \in \Hbar$ such that $y \not\in H$ and $[y, H] = 0$. 
Furthermore, by the structure of the root spaces of $L$ (cf.~\cite[Proposition 3]{CR09}),
diagonalization with respect to $\Hbar$ would give $6$ eigenspaces of dimension $4$, and one eigenspace $\Lzbar$ of dimension $12$ (where $\Hbar \subseteq \Lzbar$).
This means in particular that $L_0 = \Lzbar$ and that $y$ should have a unique eigenvalue on $L_0$. 
Since $[y, H] = 0$ and $H \subseteq L_0$, the eigenvalue of $y$ on $L_0$ must be $0$, and thus $y \in \Cent_{\Hbar}(L_0)$, implying $y \in \Cent_L(L_0)$.

However, $\Cent_L(L_0)$ is $4$-dimensional and $y_1, y_2, y_3 \in \Cent_L(L_0)$, so that (modulo linear combinations of $y_1, y_2, y_3$, and up to scalar multiples) there is only one choice for $y$:
\[
y = h_3 + h_4 + X_{\alpha_3} + X_{\alpha_9} + X_{\alpha_{12}} + X_{-\alpha_3} + X_{-\alpha_5}.
\]
Because the characteristic polynomial of $\ad_y$ is equal to $x^{16} (x+1)^4 (x^2+x+1)^8$, 
we see that $y$ is not split, and that therefore $H'$ is not a split toral subalgebra: a contradiction.
\end{proof}

This demonstrates that $H$ is an example of a split toral subalgebra that is not contained in a split toral subalgebra that is maximal among all toral subalgebras.

\section{Regular semisimple elements}\label{sec_stsa_countingrss_new}
\newcommand{\Lrss}{L_{\mathrm{rss}}}
\newcommand{\Lrssw}{L_{\mathrm{rss},w}}
\newcommand{\Lrssid}{L_{\mathrm{rss},\mathrm{id}}}

In \cite{CM06} Cohen and Murray describe an algorithm for Lang's theorem, which needs an algorithm to find split maximal toral subalgebras of Lie algebras. Although they do not claim their algorithm is valid in the characteristic $2$ case, some propositions are. We shall first introduce the concept of regular semisimple elements in order to expose some of the difficulties in characteristic $2$.

An element $x$ of a Lie algebra $L$ is called \emph{regular semisimple} if its centralizer $\Cent_L(x)$ is a maximal toral subalgebra. We denote the set of regular semisimple elements of $L$ by $\Lrss$. Moreover,
if $L$ is the Lie algebra of a group of Lie type with root datum $R$ we let $\Lrssw$ be the set of elements $x \in \Lrss$ for which there exists a $g \in G$ such that $\Cent_L(x) = H_0^g$ and $g^F g^{-1} \in T_0 \dot{w}$, where $T_0$ is the standard split maximal torus, $H_0 = \Lie(T_0)$ the corresponding split maximal toral subalgebra, $F$ denotes the Frobenius automorphism of the field, and $\dot{w}$ denotes a lift of the Weyl group element $w$. In this section we are primarily interested in split toral subalgebras, hence in $\Lrssid$.

The time analysis in \cite{CM06} uses the fact that a significant fraction of the elements in the Lie algebra is regular semisimple. In the following proposition we show that this is not always true over fields of characteristic $2$.

\begin{proposition}\label{prop_noregularsemisimple_elts}
	Let $\F$ be a field of characteristic $2$, let $R$ be a root datum of type $\mathrm A_1\SC$, $\mathrm B_2\SC$, or $\mathrm C_n\SC$ (where $n \geq 3$), and let $L$ be the Lie algebra of type $R$ over $\F$. There exist no regular semisimple elements in $L$.
\end{proposition}
\begin{proof}
	We refer to the analysis of these Lie algebras detailed in \cite[Proposition 3, Table 1]{CR09}, were it is shown that in the cases mentioned some of the root spaces are contained in the $0$-eigenspace of a split maximal toral subalgebra. This in particular implies that if $H$ is a split maximal toral subalgebra of $L$ then $H \subsetneq \Cent_L(H)$.	So suppose $x \in \Lrssid$, so that $\Cent_L(x) = H$, for some split maximal toral subalgebra $H$ of $L$. However, $x \in H$ since $x \in \Cent_L(x)$, so that $\Cent_L(x) \supseteq \Cent_L(H) \supsetneq H$, a contradiction.
\end{proof}

This shows that in some cases in characteristic $2$ there is a complete absence of regular semisimple elements.
A straightforward counting exercise \cite[Section 3.2]{roozemond10} shows that in other cases in characteristic $2$ regular semisimple elements are scarce as well, and even over small fields of odd characteristic the number of regular semisimple elements may be quite small. 

\section{The characteristic $3$ case}\label{sec_mstsa_algorithm3}
\begin{algorithm}[Split\-Max\-i\-mal\-To\-ral\-Sub\-al\-ge\-bra3]{SplitMaximalToralSubalgebra3}{Finding a split maximal toral subalgebra in char.~3}{smtsa3}
{\bf in:} \>\>\> A Lie algebra $L$ over a finite field $\F$ of characteristic $3$, \\
{\bf out:} \>\>\> A split maximal toral subalgebra $H$ of $L$.\\
\textbf{begin} \lnreset \\
\lnp \> \textbf{let} $M = L$, $H = 0$, $\phi = \mathrm{id}$, \\
\lnp \> \textbf{while} $M \neq 0$ \textbf{do}\\
     \> \> \emph{/* Base case */} \\
\lnp \> \> \textbf{if} $[M,M] = 0$ \textbf{and} $\phi(M)$ is split semisimple in $L$ \textbf{then} \\
\lnp \> \> \> \textbf{let} $H = H \cup \phi(M)$, \\
\lnp \> \> \> \textbf{return} $H$. \\
\lnp \> \> \textbf{end if} \\
     \> \> \emph{/* Try to find a new element of $H$ */} \\
\lnp \> \> \textbf{let} $\ad_M$ be the adjoint representation of $M$, \\
\lnp \> \> \textbf{let} $h$ be a random non-zero semisimple element of $M$, \\
\lnp \> \> \textbf{for each} pair of eigenvalues $(v,-v)$ of $\ad_M(h)$ \textbf{do} \\
\lnp \> \> \> \textbf{let} $s^+ \in M$ be a random element of the $v$-eigenspace of $\ad_M(h)$, \\
\lnp \> \> \> \textbf{let} $s^- \in M$ be a random element of the $-v$-eigenspace of $\ad_M(h)$, \\
\lnp \> \> \> \textbf{let} $h' = [s^+, s^-]$, \\
\lnp \> \> \> \textbf{if} $h'_L = \phi(h')$ is a split semisimple element of $L$ \textbf{then} \\
\lnp \> \> \> \> \textbf{let} $H = H \cup h'_L$, \\
\lnp \> \> \> \> \textbf{let} $M, \phi_M = \Cent_M(h')/\langle h' \rangle_M$, \\
\lnp \> \> \> \> \textbf{let} $\phi: M \rightarrow L, \phi_M \circ \phi$ the new pullback of $M$ to $L$, \\
\lnp \> \> \> \> \textbf{break for}. \\
\lnp \> \> \> \textbf{end if}. \\
\lnp \> \> \textbf{end for}. \\
\lnp \> \textbf{end while}, \\
\lnp \> \textbf{return} $H$. \\
\textbf{end}
\end{algorithm}

We first remark that the troublesome characteristic $3$ cases that Ryba \cite{Ryba07} excludes are precisely those cases discussed in \cite[Section 2]{CR09}: the Lie algebras whose root datum is either $\mathrm A_2\SC$ or $G_2$. 
The problems that arise here may be remedied relatively easily, following the observation that even though some root spaces have dimension $3$ (rather than the more desirable dimension $1$), the product of random elements of two opposite $3$-dimensional eigenspaces is often a split semisimple element. 

The algorithm is made explicit as Algorithm \ref{alg_smtsa3}: we recursively find a semisimple element $h$, find two opposite eigenspaces $S^+$ and $S^-$ of $h$, and consider random elements $s^+ \in S^+$ and $s^- \in S^-$. If $h' = [s^+, s^-]$ pulls back to a split semisimple element of $L$, we add $h'$ to the subalgebra $H$ being constructed, and continue in $\Cent_M(h')/\langle h' \rangle_M$. Throughout the algorithm, $H$ is a split toral subalgebra of $L$, $M$ is a Lie algebra being searched for new split semisimple elements, and $\phi$ is the pullback map from $M$ to $L$. Note that the image of $\phi$ is not uniquely determined, but it is determined up to addition with elements of $H$.

In Section \ref{sec_stsa_notesimpl} we present timings for Algorithm \ref{alg_smtsa3} applied to Chevalley Lie algebras in characteristic $3$. We remark that in our experience this algorithm is applicable to and yields correct results for all Chevalley Lie algebras over finite fields of any odd characteristic.

\section{The characteristic $2$ case}\label{sec_mstsa_algorithm2}

\begin{algorithm}[Split\-Max\-i\-mal\-To\-ral\-Sub\-al\-ge\-bra2]{SplitMaximalToralSubalgebra2}{Finding a split maximal toral subalgebra in char.~2}{smtsa2}
{\bf in:} \>\>\> A Lie algebra $L$ over a finite field $\F$ of characteristic $2$, \\
{\bf out:} \>\>\> A split maximal toral subalgebra $H$ of $L$.\\
\textbf{begin} \lnreset \\
\lnp \> \textbf{let} $d = \dim(\Cent_L(\hat{H}))$ where $\hat{H}$ is some Cartan subalgebra of $L$, \\
\lnp \> \textbf{let} $M = L$, $H = 0$, $\phi = \mathrm{id}$, \\
\lnp \> \textbf{while} $M \neq 0$ \textbf{and} $\dim(H) < d$ \textbf{do}\\
\lnp \> \> \textbf{if} $\dim(\Center(M)) > 0$ \textbf{then} \\
     \> \> \> \emph{/* Take out the center */} \\
\lnp \> \> \> \textbf{if} $\phi(\Center(M))$ is split semisimple \textbf{then let} $H = H \cup \phi(\Center(M))$. \\
\lnp \> \> \> \textbf{let} $M, \phi_M = M/\Center(M)$, \\
\lnp \> \> \> \textbf{let} $\phi: M \rightarrow L, \phi_M \circ \phi$. \\
\lnp \> \> \textbf{else} \\
     \> \> \> \emph{/* Try to find a new element of $H$ */} \\
\lnp \> \> \> \textbf{let} $h$ be a random non-zero semisimple element of $M$, \\
\lnp \> \> \> \textbf{if} $\phi(h)$ is split semisimple in $L$ \textbf{then} \\
\lnp \> \> \> \> \textbf{let} $h' = h$. \\
\lnp \> \> \> \textbf{else} \\
     \> \> \> \> \emph{/* Use this $h$ as input for {\sc FindSplitSemisimpleElt} */} \\
\lnp \> \> \> \> \textbf{for each} eigenvalue $v$ of $h$ \textbf{do} \\
\lnp \> \> \> \> \> \textbf{let} $V$ be the $v$-eigenspace of $h$, \\
\lnp \> \> \> \> \> \textbf{let} $h' = $\ {\sc FindSplitSemisimpleElt}($V$, $M$, $L$, $\phi$), \\
\lnp \> \> \> \> \> \textbf{if} $h' \neq $\ \textbf{fail} \textbf{then} \textbf{break}. \\
\lnp \> \> \> \> \textbf{end for}, \\
\lnp \> \> \> \textbf{end if}, \\
\lnp \> \> \> \textbf{if} $h' \neq $\ \textbf{fail} \textbf{then} \\
\lnp \> \> \> \> \textbf{let} $H = H \cup h'$, \\
\lnp \> \> \> \> \textbf{let} $M, \phi_M = \Cent_M(h')/\langle h'\rangle_M$, \\
\lnp \> \> \> \> \textbf{let} $\phi: M \rightarrow L, \phi_M \circ \phi$. \\
\lnp \> \> \> \textbf{end if}. \\
\lnp \> \> \textbf{end if}. \\
\lnp \> \textbf{end while}. \\
\textbf{end}
\end{algorithm}

\begin{algorithm}{FindSplitSemisimpleElt}{Finding a split semisimple element in an eigenspace}{smtsa_fssse}
{\bf in:} \>\>\> An eigenspace $V$ of a semisimple element of the Lie algebra $M \subseteq L$, \\
          \>\>\> and the natural pullback map $\phi: M \rightarrow L$. \\
{\bf out:} \>\>\> A split semisimple element $h \in M$, or \textbf{fail}.\\
\textbf{begin} \lnreset \\
\lnp \> \textbf{let} $S = \langle V \rangle_M$ be the subalgebra of $M$ generated by $V$, \\
\lnp \> \textbf{let} $I = ( V )_M$ be the ideal of $M$ generated by $V$, \\
\lnp \> \textbf{if} $\dim([S,S])=1$ \textbf{then} \\
     \> \> \emph{/* Case (A) */} \\
\lnp \> \> \textbf{let} $h' \in [S,S]$ be such that $[S,S] = \langle h' \rangle_\F$. \\
\lnp \> \textbf{else if} $[I,I] = I$ \textbf{and} $\dim([S,S]) \in \{2,3\}$ \textbf{then} \\
     \> \> \emph{/* Case (B) */} \\
\lnp \> \> \textbf{let} $h'$ be a random non-zero element of $[S,S]$. \\
\lnp \> \textbf{else if} $\dim(I) \neq 0$ \textbf{and} $\dim(I)$ is even \textbf{and} $\dim([I,I]) = 0 $ \\
\>\>\textbf{and} $\dim([S,S]) = 0$ \textbf{then} \\
    \> \> \emph{/* Case (C) */} \\
\lnp \> \> \textbf{find} an $h' \in M$ such that $[h',e] = e$ for all $e \in I$. \\
\lnp \> \textbf{else if} $\dim(S) = 6$ \textbf{and} $[I,I] = S$ \textbf{and} $\dim([S,S]) = 2$ \textbf{then} \\
	 \> \> \emph{/* Case (D) */} \\
\lnp \> \> \textbf{let} $h'$ be a random non-zero element of $[S,S]$. \\
\lnp \> \textbf{else if} $\dim(I) \neq 0$ \textbf{and} $\dim(I)$ is even \textbf{and} $\dim([I,I]) \neq 0 $ \\
\>\>\textbf{and} $\dim([S,S]) = 0$ \textbf{then} \\
    \> \> \emph{/* Case (E) */} \\
\lnp \> \> \textbf{find} an $h' \in I$ such that $[h',e] = e$ for all $e \in S$. \\
\lnp \> \textbf{else if} $\dim(V)$ is even \textbf{and} $\dim([S,S]) \neq 0$ \textbf{then} \\
     \> \> \emph{/* Case (F) */} \\
\lnp \> \> \textbf{let} $h'$ be a random non-zero element of $[S,S]$ \\
\lnp \> \textbf{end if}, \\
\lnp \> \textbf{if} $h'$ is defined and $\phi(h')$ is a split semisimple element of $L$ \textbf{then} \\
\lnp \> \> \textbf{return} $h'$.\\
\lnp \> \textbf{else} \\
\lnp \> \> \textbf{return} \textbf{fail}.\\
\lnp \> \textbf{end if}. \\
\textbf{end}
\end{algorithm}

Proposition \ref{prop_noregularsemisimple_elts} indicates that the approach for finding
split maximal toral subalgebras described by Cohen and Murray \cite[Section 5]{CM06} will not in general work in the 
cases covered by the proposition: there simply do not exist any regular semisimple elements in the Lie algebra.
Moreover, their algorithm relies on the fact that root spaces are $1$-dimensional,
something that is not true over characteristic $2$ \cite[Proposition 3]{CR09}.

Ryba explicitly notes \cite[Section 9]{Ryba07} that the algorithm he describes
is not easily extended to work over fields of characteristic $2$,
largely because of similar problems.
Finally, the counterexample presented in Section \ref{sec_scsa_C4} suggests that 
algorithms for finding split maximal toral subalgebras run the risk of 
descending into a split toral subalgebra that is not in a split maximal toral subalgebra.

\bigskip

In this section we describe a heuristic Las Vegas type algorithm for finding split maximal toral subalgebras in characteristic $2$.
Unfortunately, we have no bound on the probability that it completes successfully,
and therefore no estimate of the runtime. However, we do provide the intuition behind the
design of the algorithm (in the remainder of this section) and we show that the implementation
is successful (we give timings in Section \ref{sec_stsa_notesimpl}).

For the remainder of this section we let $L$ be the Lie algebra of a split simple algebraic group defined over a finite field $\F$ of characteristic $2$,
and we assume $L$ to be given as a structure constant algebra. 
The goal of the algorithm described is to find a split maximal toral subalgebra $H$ of $L$.

The general principle is given in Algorithm \ref{alg_smtsa2}. This algorithm repeatedly tries to find a split semisimple element $h' \in M$ (initially $M = L$), and then recursively continues the search in $\Cent_M(h')/\langle h' \rangle_M$. It attempts to find such split semisimple elements by taking a random non-zero semisimple element $h$, and producing a random split semisimple element using suitable eigenspaces of $h$. The latter process is described in Algorithm \ref{alg_smtsa_fssse}.

\begin{table}
	\[
	\renewcommand\arraystretch{1.3}%
	\begin{array}{ll|lllll|l}
		R & \mbox{Mult} & S & [S,S] & [S,S] \cap H & I & [I,I] & \mbox{Soln}\\
		\hline
		\hline
		\mathrm A_1\Ad & 2 & 2 & 0 & 0 & 2 & 2 & \mbox{(C)} \cr
		\mathrm A_1\SC & \mathbf{2} & 3 & 1 & 1 & 3 & 1 & \mbox{(A)}\cr
		\mathrm A_3\SC & 4^3 & 6 & 2 & 2 & L & I & \mbox{(B)}\cr
		\mathrm A_3^{(2)} & 4^3 & 5 & 1 & 1 & L-1 & I & \mbox{(A)}\cr
		\mathrm B_2\Ad & 2^2 & 2 & 0 & 0 & 4 & 0 & \mbox{(C)} \cr
		       \;\lfloor & 4 & 5 & 1 & 1 & 9 & 5 & \mbox{(A)}\cr
		\mathrm B_n\Ad\; (n \geq 3) & 2^n & 2 & 0 & 0 & 2n & 0 & \mbox{(C)} \cr
		       \;\lfloor          & 4^{n \choose 2} & 5 & 1 & 1 & L-1 & I & \mbox{(A)}\cr
		\mathrm B_2\SC & \mathbf{4} & 6 & 2 & 2 & L & 6 & \mbox{(D)}\cr
		     \;\lfloor & 4 & 5 & 1 & 1 & 5 & 1 & \mbox{(A)}\cr
		\mathrm B_3\SC & 6^3 & 8 & 2 & 2 & L & I & \mbox{(B)}\cr
		\mathrm B_4\SC & 2^4 & 3 & 1 & 1 & 9 & 1 & \mbox{(A)}\cr
		     \;\lfloor & 8^3 & 11 & 3 & 3 & L & I & \mbox{(B)}\cr
		\mathrm B_n\SC\; (n \geq 5) & 2^n & 3 & 1 & 1 & 2n+1 & 1 & \mbox{(A)}\cr
		     \;\lfloor & 4^{n \choose 2} & 6 & 2 & 2 & L & I & \mbox{(B)}\cr
		\mathrm C_n\Ad\; (n \geq 3) & 2n & 3n-1 & n-1 & n-1 & L & & \mbox{(F)} \cr
		     \;\lfloor            & 2^{n(n-1)} & 3 & 1 & 1 &  & I & \mbox{(A)}\cr
		\mathrm C_n\SC\; (n \geq 3) & \mathbf{2n} & 3n & n & n & L &  & \mbox{(F)}\cr
		     \;\lfloor            & 4^{n \choose 2} & 5 & 1 & 1 &  & I & \mbox{(A)}\cr
		\mathrm D_4\SC & 8^3 & 11 & 3 & 3 & L & I & \mbox{(B)}\cr
		\mathrm D_4^{(1),(n),(n-1)} & 4^6 & 5 & 1 & 1 & L-1 & I & \mbox{(A)}\cr
		\mathrm D_n\SC\; (n \geq 5) & 4^{n \choose 2} & 6 & 2 & 2 & L & I & \mbox{(B)}\cr
		\mathrm D_n^{(1)}\; (n \geq 5) & 4^{n \choose 2} & 5 & 1 & 1 & L-1 & I & \mbox{(A)}\cr
		\mathrm F_4 & 2^{12} & 3 & 1 & 1 & 26 & I & \mbox{(A)}\cr
		    \;\lfloor & 8^3 & 11 & 3 & 3 & L & I & \mbox{(B)}\cr
		\mathrm G_2 & 4^3 & 5 & 1 & 1 & L & I & \mbox{(A)} 
	\end{array}
	\]
	\caption{Eigenspaces, their subalgebras, and their ideals in characteristic $2$}
	\label{tab_stsa_eigenspaces}
\end{table}

\bigskip

In order to clarify Algorithm \ref{alg_smtsa_fssse} we let $R$ be an irreducible root datum, $\F$ a field of characteristic $2$, and $L$ the Lie algebra of type $R=(X,\Phi,Y,\Phi\ch)$ over $\F$. Recall the definition of root spaces $L_\alpha$ from Section \ref{sec_introduction}.
Observe first of all that, since $\chr(\F) = 2$, the root spaces $L_{\alpha}$ and $L_{-\alpha}$ coincide for all $\alpha \in \Phi$. This implies that $\alpha\ch \in [L_\alpha, L_\alpha]$, prompting us to consider $[S,S]$ in line 4 of Algorithm \ref{alg_smtsa_fssse}.

We justify the choices for the various other cases in this algorithm using the data in Table \ref{tab_stsa_eigenspaces}. 
In the first column that table contains the root data $R$ that are proved to have multidimensional root spaces over fields of characteristic $2$ (cf.~\cite[Proposition 3]{CR09}). The hook symbols in the first column indicate a case spread over multiple rows, e.g., the 5th and 6th row both concern the case $\mathrm B_2\Ad$.
The second column labeled ``$\mbox{Mult}$'' contains the dimensions and multiplicities of each of these root spaces
in the same notation used in \cite{CR09}, i.e., $d^k$ signifies $k$ distinct eigenspaces of dimension $d$ and such a $\mathbf{d^k}$ printed in boldface indicates an eigenspace that occurs with eigenvalue $0$.

To clarify the other columns we let $V$ be one of the eigenspaces mentioned (e.g., for the eighth line of the table $L = \mathrm B_n\Ad(\F)$ and $V$ is one of the $4$-dimensional (long) root spaces). Then we let $S = \langle V \rangle_L$ be the subalgebra generated by $V$ and $I = ( V )_L$ the ideal generated by $V$.
Now the third column contains the dimension of $S$, the fourth column the dimension of $[S,S]$ and the fifth the dimension of $[S,S] \cap H$. 
The sixth column contains the dimension of $I$, or ``$L$'' if $I = L$, or ``$L-1$'' if $I$ is a codimension one ideal of $L$, and the seventh column contains the dimension of $[I,I]$, or ``$I$'' if $[I,I]=I$.
Finally, the eighth column shows which of the cases of Algorithm \ref{alg_smtsa_fssse} is based on this type of root space.

The case distinction in Algorithm \ref{alg_smtsa_fssse} is based on the observations in 
Table \ref{tab_stsa_eigenspaces} in the following manner.
\begin{enumerate}
	\renewcommand{\theenumi}{(\Alph{enumi})}
	\item In each of the cases where $\dim([S,S]) = 1$ we have $[S,S] \subseteq H$, prompting us to take $h$ to be a basis element of $[S,S]$. Note that this case also applies if $V$ corresponds to the direct sum of several Lie algebras of type $\mathrm A_1\SC$.
	\item In the cases where $[I,I]=I$ and $\dim([S,S]) \in \{2,3\}$ we also have $[S,S] \subseteq H$, so that a random non-zero element of $[S,S]$ seems a good candidate.
	\item In the cases where $\dim([I,I]) = \dim([S,S]) = 0$ the best candidate we can find is an element $h \in M$ that acts on $I$ the way a split semisimple element should. Note that this case also applies if $V$ corresponds to the direct sum of several Lie algebras of type $\mathrm A_1\Ad$.
	\item In the cases where $\dim(S)=6$ (prime example being the long roots in $\mathrm B_2\SC$) we also pick a random non-zero element of $[S,S]$ as candidate.
	\item This case is special since it does not occur in Table \ref{tab_stsa_eigenspaces}. It is however needed to successfully complete the search for a split maximal toral subalgebra if $L$ is of type $\mathrm C_n\SC$. The solution is similar to that of case (C).
	\item This case is needed for Lie algebras of type $\mathrm C_n$, where again $[S,S] \subseteq H$, but the dimension of $[S,S]$ can be as large as $\dim(H)$. Again, we pick a random non-zero element of $[S,S]$ as candidate.
\end{enumerate}

Given a reductive Lie algebra $L$, the \emph{reductive rank} of $L$ is the rank of the root datum of $L$, or, equivalently, the dimension of its split maximal toral subalgebra. In the first line of Algorithm \ref{alg_smtsa2} we define $d$ to be $\dim(\Cent_L(\hat{H}))$ for some Cartan subalgebra $\hat{H}$. This integer is the dimension of $H$ we are aiming for throughout the algorithm, in effect claiming that the reductive rank of $L$ must be $d$. The validity of this claim is asserted by the following lemma.
\begin{lemma}
Let $L$ be a reductive Lie algebra over a field $\F$ whose root datum $R = (X,\Phi,Y,\Phi\ch)$ is irreducible, and let $d = \rk(R)$ be its reductive rank.
If $\hat{H}$ is a (not necessarily split) Cartan subalgebra of $L$ then $\dim(\Cent_L(\hat{H})) = d$.
\end{lemma}
\begin{proof}
By \cite[Proposition 15.2]{Hum67} we have that $\hat{H} = \Cent_L(H)$ for some maximal toral subalgebra $H$ of $L$. 
Since $H$ is a maximal toral subalgebra, we have $[H,H] = 0$ so that $H \subseteq \Cent_L(H)$.
If $H = \Cent_L(H)$, then
\[
	\dim(\Cent_L(\hat{H})) = \dim(\Cent_L(\Cent_L(H))) = \dim(H) = d,
\]
so the only case left to consider is when $H \subsetneq \Cent_L(H)$.

For a suitable field extension, $\F' \supseteq \F$ say, $H_{\F'} = H \otimes \F'$ is split, so $H_{\F'}$ diagonalises $L_{\F'} = L \otimes \F'$ into root spaces $L_\alpha$. Since we assumed $H \subsetneq \Cent_L(H)$, there exists an $L_\alpha$ such that $[L_\alpha, H_{\F'}] = 0$, so there is a root whose eigenvalue under $H_{\F'}$ is $0$. By \cite[Proposition 3]{CR09} that means $R$ is of type $\mathrm C_n\SC$, where $n \geq 2$ (note that this includes $\mathrm B_2\SC$).

So assume $L$ is of type $\mathrm C_n\SC$, with $n \geq 2$, with a Chevalley basis $\{ X_\alpha \mid \alpha \in \Phi\} \cup \{ h_i \mid i =1, \ldots, n \}$, where $h_i \in H$, and let $S = \Cent_L(H)$. Inspection of this family of Lie algebras shows that $\dim(S) = 3n$ and $S = H \cup \langle X_\alpha \mid \alpha \in \Phi^{\mathrm{long}} \rangle$. Moreover, for all $\alpha \in \Phi^{\mathrm{long}}$ we have $[X_\alpha, X_{-\alpha}] \neq 0$ and, as is always the case, $[X_\alpha, X_{-\alpha}] \in H$.

Now suppose $x \in \Cent_L(S)$. Then, since $x \in L$, we have
\[
x = \sum_{\alpha \in \Phi} c_\alpha X_\alpha + \sum_{i=1}^n t_i h_i,
\]
for some $c_\alpha$, $t_i$. If $c_\alpha \neq 0$ for some $\alpha \in \Phi$, then $[x,H] \neq 0$ if $\alpha$ is short, and $[x, X_{-\alpha}] \neq 0$ if $\alpha$ is long (since for all $\beta,\gamma \in \Phi^{\mathrm{long}}$ we have $[X_\beta, H] = 0$, and $[X_\beta,X_\gamma] = 0$ unless $\beta=-\gamma$), in both cases contradicting $[x,S] = 0$. This shows that $x \in H$, and therefore $\Cent_L(S) = H$ and $\dim(\Cent_L(\hat{H})) = \dim(H) = d$, proving the lemma.
\end{proof}
The lemma immediately generalises to the case where $R$ is not irreducible.
We note that a Cartan subalgebra can be computed efficiently \cite[Section 3.2]{deGraaf00}. In our experience, in fact, the time it takes to compute a Cartan subalgebra is negligible compared to the overall time taken by Algorithm \ref{alg_smtsa2}.

We end this section with a number of remarks on the implementation of the algorithm.
Firstly, from the manner in which the algorithm is specified we can conclude that it may run for an infinite time. Indeed, $M$ only decreases in dimension if a new split semisimple element is found and such an element does not always exist, as shown in Section \ref{sec_scsa_C4}.
Also, in various cases the algorithm {\sc FindSplitSemisimpleElt} will fail to return a split semisimple $h$, due to the simple fact that $S$ is not of a suitable type or the candidate $h$ turns out not to be split.
In the implementation of this algorithm these problems are remedied by limiting the number of random tries allowed for each $M$ in line $9$ of {\sc SplitMaximalToralSubalgebra} to some finite number. If after that number of tries no new $H$ was found, the algorithm terminates and returns \textbf{fail}. 

Secondly, note that the influence of the size of the field on the performance of the algorithm is twofold. Firstly, the smaller the field, the higher the probability of finding split semisimple elements in Algorithm \ref{alg_smtsa_fssse}. On the other hand, the bigger the field, the higher the probability that the random semisimple elements picked in Algorithm \ref{alg_smtsa2} have eigenspaces of small dimension. This dichotomy yields an algorithm whose performance is acceptable both over small and over larger fields.
We will, of course, in general see a decreasing performance of the algorithm as the size of $\F$ increases, simply because field arithmetic and therefore Lie algebra arithmetic slow down.

\section{Implementation and performance}\label{sec_stsa_notesimpl}

\begin{table}
\ \\

\vspace{-10mm} \hspace{-5mm}
\begin{tabular}{p{65mm}p{65mm}}	
\[ \begin{array}{l...} 
R & \multicolumn{1}{c}{\GF(3)} & \multicolumn{1}{c}{\GF(3^6)} & \multicolumn{1}{c}{\GF(3^{10})} \cr
\hline
\hline
\mathrm A_{1}^{\mathrm{ad}} & 0 & 0 & 0 \cr
\mathrm A_{1}^{\mathrm{SC}} & 0 & 0 & 0 \cr
\mathrm A_{2}^{\mathrm{ad}} & 0 & 0 & 0 \cr
\mathrm A_{2}^{\mathrm{SC}} & 0 & 0 & 0 \cr
\mathrm A_{3}^{\mathrm{ad}} & 0 & 0 & 0 \cr
\mathrm A_{3}^{(2)}         & 0 & 0 & 0 \cr
\mathrm A_{3}^{\mathrm{SC}} & 0 & 0 & 0 \cr
\mathrm A_{4}^{\mathrm{ad}} & 0.1 & 0.1 & 0.1 \cr
\mathrm A_{4}^{\mathrm{SC}} & 0.1 & 0.1 & 0.1 \cr
\mathrm A_{5}^{\mathrm{ad}} & 0.2 & 0.5 & 0.5 \cr
\mathrm A_{5}^{(2)}         & 0.3 & 0.5 & 0.6 \cr
\mathrm A_{5}^{(3)}         & 0.2 & 0.4 & 0.4 \cr
\mathrm A_{5}^{\mathrm{SC}} & 0.2 & 0.4 & 0.3 \cr
\mathrm A_{6}^{\mathrm{ad}} & 0.3 & 1.1 & 1.2 \cr
\mathrm A_{6}^{\mathrm{SC}} & 0.6 & 1.1 & 1.1 \cr
\mathrm A_{7}^{\mathrm{ad}} & 1.3 & 3.3 & 3.3 \cr
\mathrm A_{7}^{(2)}         & 1.7 & 3 & 4.1 \cr
\mathrm A_{7}^{(4)}         & 1.2 & 3.1 & 4.3 \cr
\mathrm A_{7}^{\mathrm{SC}} & 1.7 & 3.4 & 4.9 \cr
\mathrm A_{8}^{\mathrm{ad}} & 5.2 & 21 & 17 \cr
\mathrm A_{8}^{(3)}         & 5.6 & 17 & 21 \cr
\mathrm A_{8}^{\mathrm{SC}} & 3 & 14 & 10 \cr
\mathrm B_{2}^{\mathrm{ad}} & 0 & 0 & 0 \cr
\mathrm B_{2}^{\mathrm{SC}} & 0 & 0 & 0 \cr
\mathrm B_{3}^{\mathrm{ad}} & 0 & 0 & 0 \cr
\mathrm B_{3}^{\mathrm{SC}} & 0 & 0 & 0 \cr
\mathrm B_{4}^{\mathrm{ad}} & 0.1 & 0.2 & 0.3 \cr
\mathrm B_{4}^{\mathrm{SC}} & 0.1 & 0.2 & 0.3 \cr
\mathrm B_{5}^{\mathrm{ad}} & 0.7 & 1.5 & 1.8 \cr
\mathrm B_{5}^{\mathrm{SC}} & 0.5 & 1.3 & 1.4 \cr
\mathrm B_{6}^{\mathrm{ad}} & 2.6 & 8.1 & 8.2 \cr
\mathrm B_{6}^{\mathrm{SC}} & 2.4 & 6.2 & 6.9 \cr
\mathrm B_{7}^{\mathrm{ad}} & 7.3 & 23 & 25 \cr
\mathrm B_{7}^{\mathrm{SC}} & 7.9 & 23 & 29 \cr
\mathrm B_{8}^{\mathrm{ad}} & 37 & 74 & 102 \cr
\mathrm B_{8}^{\mathrm{SC}} & 25 & 70 & 85 \cr
\mathrm C_{3}^{\mathrm{ad}} & 0 & 0 & 0 \cr
\mathrm C_{3}^{\mathrm{SC}} & 0 & 0 & 0 
\end{array} \] &
\[ \begin{array}{l...} 
R & \multicolumn{1}{c}{\GF(3)} & \multicolumn{1}{c}{\GF(3^6)} & \multicolumn{1}{c}{\GF(3^{10})} \cr
\hline
\hline
\mathrm C_{4}^{\mathrm{ad}} & 0.1 & 0.2 & 0.3 \cr
\mathrm C_{4}^{\mathrm{SC}} & 0.1 & 0.2 & 0.3 \cr
\mathrm C_{5}^{\mathrm{ad}} & 0.6 & 1.4 & 1.7 \cr
\mathrm C_{5}^{\mathrm{SC}} & 0.6 & 1.6 & 1.5 \cr
\mathrm C_{6}^{\mathrm{ad}} & 2.6 & 7.5 & 7.9 \cr
\mathrm C_{6}^{\mathrm{SC}} & 3 & 6.5 & 8 \cr
\mathrm C_{7}^{\mathrm{ad}} & 7.3 & 26 & 30 \cr
\mathrm C_{7}^{\mathrm{SC}} & 7 & 21 & 26 \cr
\mathrm C_{8}^{\mathrm{ad}} & 29 & 73 & 109 \cr
\mathrm C_{8}^{\mathrm{SC}} & 22 & 86 & 110 \cr
\mathrm D_{4}^{\mathrm{ad}} & 0.1 & 0.1 & 0.1 \cr
\mathrm D_{4}^{(1)}         & 0 & 0.1 & 0.1 \cr
\mathrm D_{4}^{(n-1)}       & 0.1 & 0.1 & 0.1 \cr
\mathrm D_{4}^{(n)}         & 0.1 & 0.1 & 0.1 \cr
\mathrm D_{4}^{\mathrm{SC}} & 0.1 & 0.2 & 0.1 \cr
\mathrm D_{5}^{\mathrm{ad}} & 0.3 & 0.8 & 0.9 \cr
\mathrm D_{5}^{(1)}         & 0.3 & 0.7 & 0.8 \cr
\mathrm D_{5}^{\mathrm{SC}} & 0.4 & 0.6 & 0.7 \cr
\mathrm D_{6}^{\mathrm{ad}} & 1.5 & 4.9 & 5 \cr
\mathrm D_{6}^{(1)}         & 1.3 & 3.4 & 3.7 \cr
\mathrm D_{6}^{(n-1)}       & 1.6 & 3.9 & 3.9 \cr
\mathrm D_{6}^{(n)}         & 1.2 & 3.7 & 3.7 \cr
\mathrm D_{6}^{\mathrm{SC}} & 1.1 & 4.6 & 4.4 \cr
\mathrm D_{7}^{\mathrm{ad}} & 3.5 & 16 & 19 \cr
\mathrm D_{7}^{(1)}         & 6 & 17 & 17 \cr
\mathrm D_{7}^{\mathrm{SC}} & 6.7 & 19 & 21 \cr
\mathrm D_{8}^{\mathrm{ad}} & 12 & 54 & 71 \cr
\mathrm D_{8}^{(1)}         & 14 & 48 & 64 \cr
\mathrm D_{8}^{(n-1)}       & 10 & 53 & 64 \cr
\mathrm D_{8}^{(n)}         & 12 & 49 & 64 \cr
\mathrm D_{8}^{\mathrm{SC}} & 12 & 73 & 64 \cr
\mathrm E_{6}^{\mathrm{ad}} & 3 & 12 & 16 \cr
\mathrm E_{6}^{\mathrm{SC}} & 2.3 & 5.9 & 8 \cr
\mathrm E_{7}^{\mathrm{ad}} & 14 & 66 & 70 \cr
\mathrm E_{7}^{\mathrm{SC}} & 13 & 88 & 72 \cr
\mathrm E_{8} & 132 & 1269 & 1213 \cr
\mathrm F_{4} & 0.4 & 0.8 & 0.8 \cr
\mathrm G_{2} & 0 & 0 & 0 
\end{array} \]
\end{tabular}
\caption{Runtimes for {\sc SplitMaximalToralSubalgebra3}}\label{tab_stsa3_runtimes_byalg}
\end{table}

\begin{table}
\ \\

\vspace{-10mm} \hspace{-5mm}
\begin{tabular}{p{65mm}p{65mm}}	
\[ \begin{array}{l...} 
R & \multicolumn{1}{c}{\GF(2)} & \multicolumn{1}{c}{\GF(2^6)} & \multicolumn{1}{c}{\GF(2^{10})} \cr
\hline
\hline
\mathrm A_{1}^{\mathrm{ad}} & 0 & 0 & 0 \cr
\mathrm A_{1}^{\mathrm{SC}} & 0 & 0 & 0 \cr
\mathrm A_{2}^{\mathrm{ad}} & 0 & 0 & 0 \cr
\mathrm A_{2}^{\mathrm{SC}} & 0 & 0 & 0 \cr
\mathrm A_{3}^{\mathrm{ad}} & 0.1 & 0.1 & 0.1 \cr
\mathrm A_{3}^{(2)}         & 0 & 0.1 & 0.1 \cr
\mathrm A_{3}^{\mathrm{SC}} & 0 & 0.1 & 0.1 \cr
\mathrm A_{4}^{\mathrm{ad}} & 0.1 & 0.4 & 0.3 \cr
\mathrm A_{4}^{\mathrm{SC}} & 0.1 & 0.4 & 0.3 \cr
\mathrm A_{5}^{\mathrm{ad}} & 0.4 & 1.8 & 1.2 \cr
\mathrm A_{5}^{(2)}         & 0.4 & 2.2 & 1.7 \cr
\mathrm A_{5}^{(3)}         & 0.4 & 1.8 & 1.2 \cr
\mathrm A_{5}^{\mathrm{SC}} & 0.4 & 2.3 & 1.6 \cr
\mathrm A_{6}^{\mathrm{ad}} & 1.1 & 6.2 & 4.9 \cr
\mathrm A_{6}^{\mathrm{SC}} & 1 & 6.2 & 4.1 \cr
\mathrm A_{7}^{\mathrm{ad}} & 4.3 & 18 & 13 \cr
\mathrm A_{7}^{(2)}         & 4.3 & 20 & 17 \cr
\mathrm A_{7}^{(4)}         & 3.4 & 20 & 16 \cr
\mathrm A_{7}^{\mathrm{SC}} & 3.8 & 22 & 14 \cr
\mathrm A_{8}^{\mathrm{ad}} & 9.9 & 51 & 35 \cr
\mathrm A_{8}^{(3)}         & 9.8 & 46 & 43 \cr
\mathrm A_{8}^{\mathrm{SC}} & 9.6 & 50 & 37 \cr
\mathrm B_{2}^{\mathrm{ad}} & 0 & 0 & 0 \cr
\mathrm B_{2}^{\mathrm{SC}} & 0 & 0.1 & 0.1 \cr
\mathrm B_{3}^{\mathrm{ad}} & 0.1 & 0.2 & 0.2 \cr
\mathrm B_{3}^{\mathrm{SC}} & 0.1 & 0.3 & 0.3 \cr
\mathrm B_{4}^{\mathrm{ad}} & 0.3 & 1.6 & 1.4 \cr
\mathrm B_{4}^{\mathrm{SC}} & 0.4 & 2.4 & 1.9 \cr
\mathrm B_{5}^{\mathrm{ad}} & 1.6 & 8.9 & 7.3 \cr
\mathrm B_{5}^{\mathrm{SC}} & 1.7 & 10 & 7.1 \cr
\mathrm B_{6}^{\mathrm{ad}} & 6 & 38 & 31 \cr
\mathrm B_{6}^{\mathrm{SC}} & 8.1 & 45 & 28 \cr
\mathrm B_{7}^{\mathrm{ad}} & 20 & 121 & 89 \cr
\mathrm B_{7}^{\mathrm{SC}} & 22 & 128 & 114 \cr
\mathrm B_{8}^{\mathrm{ad}} & 70 & 353 & 319 \cr
\mathrm B_{8}^{\mathrm{SC}} & 77 & 405 & 311 \cr
\mathrm C_{3}^{\mathrm{ad}} & 0.1 & 0.4 & 0.3 \cr
\mathrm C_{3}^{\mathrm{SC}} & 0.1 & 0.4 & 0.3 
\end{array} \] &
\[ \begin{array}{l...} 
R & \multicolumn{1}{c}{\GF(2)} & \multicolumn{1}{c}{\GF(2^6)} & \multicolumn{1}{c}{\GF(2^{10})} \cr
\hline
\hline
\mathrm C_{4}^{\mathrm{ad}} & 0.8 & 2.4 & 1.7 \cr
\mathrm C_{4}^{\mathrm{SC}} & 0.3 & 2 & 1.8 \cr
\mathrm C_{5}^{\mathrm{ad}} & 4.7 & 20 & 8.1 \cr
\mathrm C_{5}^{\mathrm{SC}} & 1.6 & 12 & 8.5 \cr
\mathrm C_{6}^{\mathrm{ad}} & 35 & 111 & 50 \cr
\mathrm C_{6}^{\mathrm{SC}} & 6.7 & 43 & 41 \cr
\mathrm C_{7}^{\mathrm{ad}} & 97 & 244 & 218 \cr
\mathrm C_{7}^{\mathrm{SC}} & 21 & 156 & 129 \cr
\mathrm C_{8}^{\mathrm{ad}} & 375 & 1059 & 1099 \cr
\mathrm C_{8}^{\mathrm{SC}} & 67 & 510 & 472 \cr
\mathrm D_{4}^{\mathrm{ad}} & 0.4 & 0.8 & 0.5 \cr
\mathrm D_{4}^{(1)}         & 0.3 & 0.8 & 0.8 \cr
\mathrm D_{4}^{(n-1)}       & 0.1 & 0.8 & 0.7 \cr
\mathrm D_{4}^{(n)}         & 0.2 & 0.9 & 0.7 \cr
\mathrm D_{4}^{\mathrm{SC}} & 0.2 & 1.1 & 0.8 \cr
\mathrm D_{5}^{\mathrm{ad}} & 0.9 & 4.2 & 3.3 \cr
\mathrm D_{5}^{(1)}         & 0.8 & 5.2 & 3.6 \cr
\mathrm D_{5}^{\mathrm{SC}} & 0.9 & 4.6 & 4.1 \cr
\mathrm D_{6}^{\mathrm{ad}} & 5.8 & 20 & 16 \cr
\mathrm D_{6}^{(1)}         & 4 & 22 & 14 \cr
\mathrm D_{6}^{(n-1)}       & 6.4 & 24 & 15 \cr
\mathrm D_{6}^{(n)}         & 5.4 & 23 & 15 \cr
\mathrm D_{6}^{\mathrm{SC}} & 5.3 & 25 & 20 \cr
\mathrm D_{7}^{\mathrm{ad}} & 24 & 98 & 53 \cr
\mathrm D_{7}^{(1)}         & 14 & 79 & 53 \cr
\mathrm D_{7}^{\mathrm{SC}} & 12 & 78 & 62 \cr
\mathrm D_{8}^{\mathrm{ad}} & 57 & 293 & 166 \cr
\mathrm D_{8}^{(1)}         & 55 & 218 & 184 \cr
\mathrm D_{8}^{(n-1)}       & 120 & 266 & 224 \cr
\mathrm D_{8}^{(n)}         & 66 & 322 & 162 \cr
\mathrm D_{8}^{\mathrm{SC}} & 54 & 314 & 228 \cr
\mathrm E_{6}^{\mathrm{ad}} & 5.7 & 33 & 31 \cr
\mathrm E_{6}^{\mathrm{SC}} & 5.2 & 33 & 31 \cr
\mathrm E_{7}^{\mathrm{ad}} & 65 & 268 & 258 \cr
\mathrm E_{7}^{\mathrm{SC}} & 72 & 300 & 217 \cr
\mathrm E_{8} & 492 & 4261 & 3795 \cr
\mathrm F_{4} & 1.5 & 8.2 & 5.6 \cr
\mathrm G_{2} & 0 & 0.1 & 0 
\end{array} \]
\end{tabular}
\caption{Runtimes for {\sc SplitMaximalToralSubalgebra2}}\label{tab_stsa2_runtimes_byalg}
\end{table}

We have implemented the algorithms discussed in the {\Magma} computer algebra system \cite{Magma}, and comment on the performance of the implementation in this section.
We present timings of runs of the {\sc SplitMaximalToralSubalgebra2} and {\sc SplitMaximalToralSubalgebra3} algorithms on Lie algebras of split simple algebraic groups over six different fields. In every case the input of the algorithm was the appropriate Chevalley Lie algebra, given as a multiplication table on a uniformly random basis. In Table \ref{tab_stsa3_runtimes_byalg} and in Figure \ref{fig_stsa_runtimes_byalg3}, the algorithm was run for Lie algebras up to rank $8$, over fields of size $3$, $3^6$, and $3^{10}$; in Table \ref{tab_stsa2_runtimes_byalg} and in Figure \ref{fig_stsa_runtimes_byalg}, for Lie algebras up to rank $8$, over fields of size $2$, $2^6$, and $2^{10}$; and in Figures \ref{fig_stsa_runtimes_byfld3} and \ref{fig_stsa_runtimes_byfld} for seven different Lie algebras, varying the size of the field between $2$ and $2^{40}$ and between $3$ and $3^{40}$, respectively. All timings are in seconds and were created using a development version of {\Magma}, 2.18, on a 2GHz AMD processor.

We remark that in a sense the timings presented represent the worst possible: because the multiplication table is given on a uniformly random basis it is very dense, making multiplication an exceptionally expensive operation. In practice when a Lie algebra arises from other computations, the multiplication table could be much sparser and the algorithm therefore much faster. 

Despite the fact that we have not commented on the computational complexity of our algorithm, the four graphs give some indication. In particular, Figures \ref{fig_stsa_runtimes_byalg3} and \ref{fig_stsa_runtimes_byalg} suggest a dependence on the rank $n$ of approximately $\bigO(n^8)$ and Figures \ref{fig_stsa_runtimes_byfld3} and \ref{fig_stsa_runtimes_byfld} suggest a linear dependence on the logarithm of the size of the field. This leads to an approximate complexity of $\bigO(n^8 \log(q))$, where $n$ is the reductive rank of the Lie algebra and $q$ the size of the field. This is not as bad as it may seem at first sight, as a single Lie multiplication already takes $\bigO(n^6 \log(q))$ time.

\section*{Acknowledgements}
The author would like to thank Arjeh Cohen for numerous fruitful discussions on this topic and the anonymous reviewers for their thorough evaluation and their valuable comments.

\begin{figure}
	\includegraphics[width=135mm]{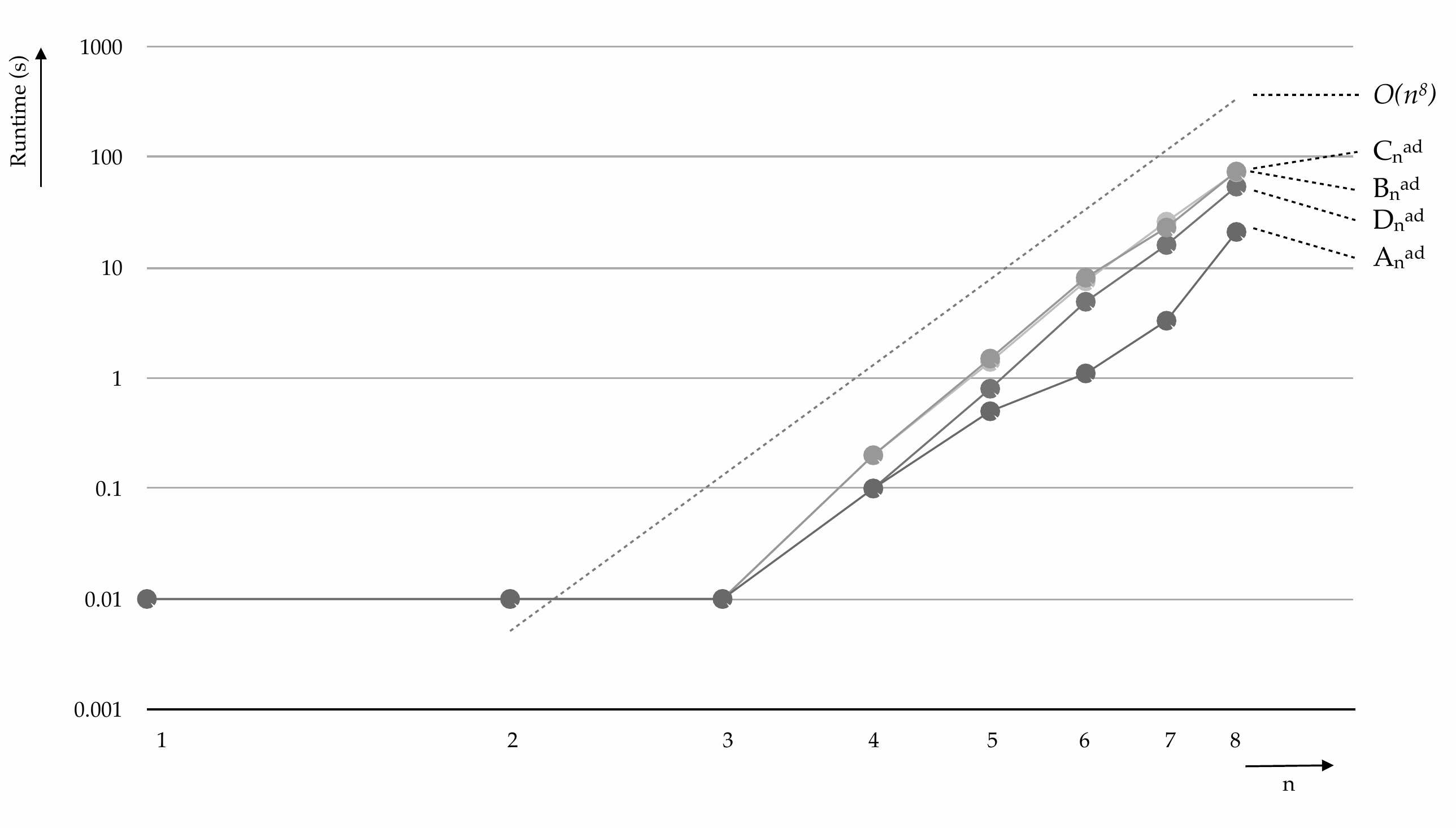}
	\caption{Runtimes for {\sc SplitMaximalToralSubalgebra3} for $\F = \GF(3^6)$}\label{fig_stsa_runtimes_byalg3}
\end{figure}

\begin{figure}
	\includegraphics[width=145mm]{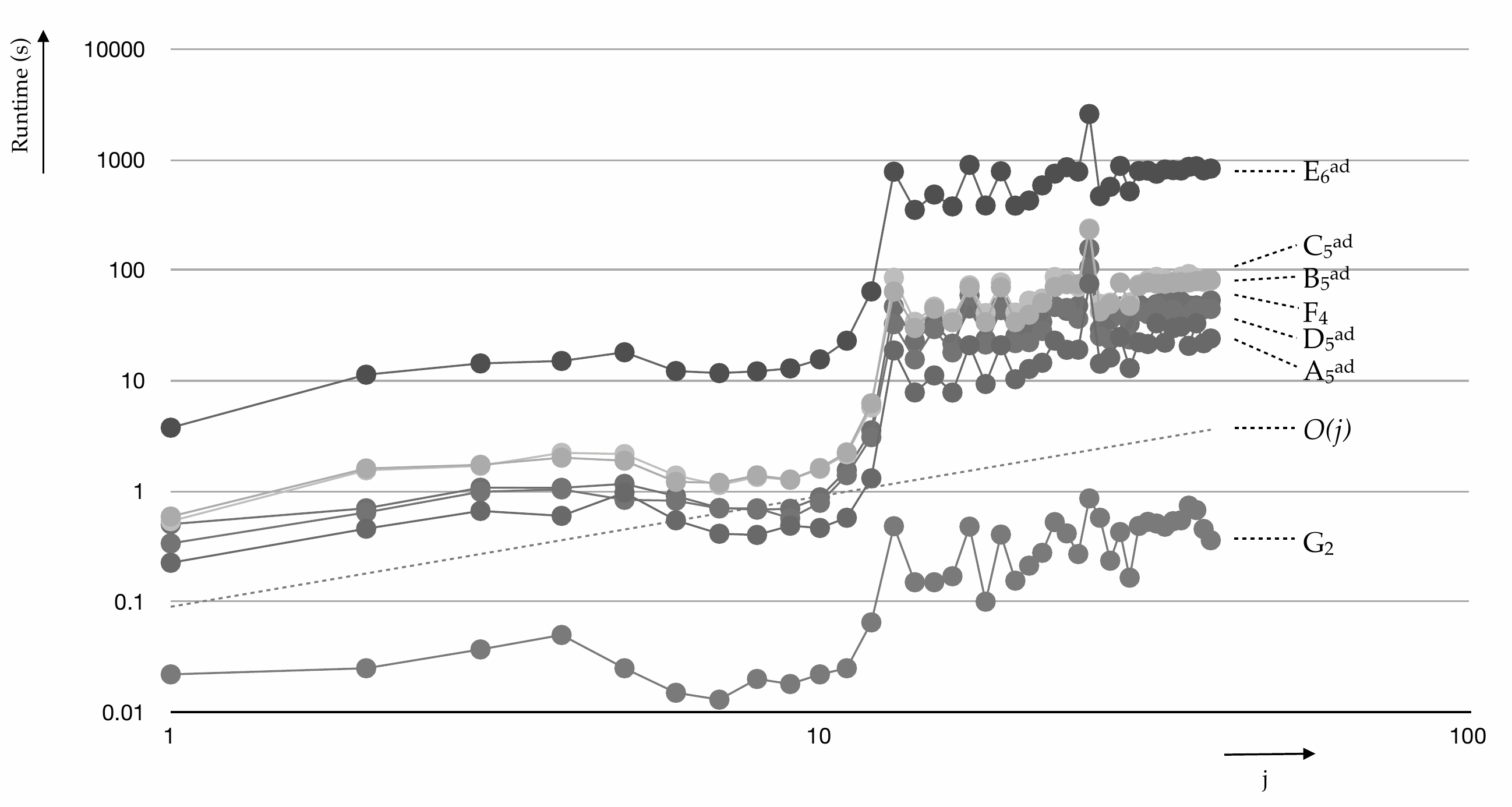}
	\caption{Runtimes for {\sc SplitMaximalToralSubalgebra3} for $\F = \GF(3^j)$, $1 \leq j \leq 40$}\label{fig_stsa_runtimes_byfld3}
\end{figure}

\begin{figure}
	\includegraphics[width=135mm]{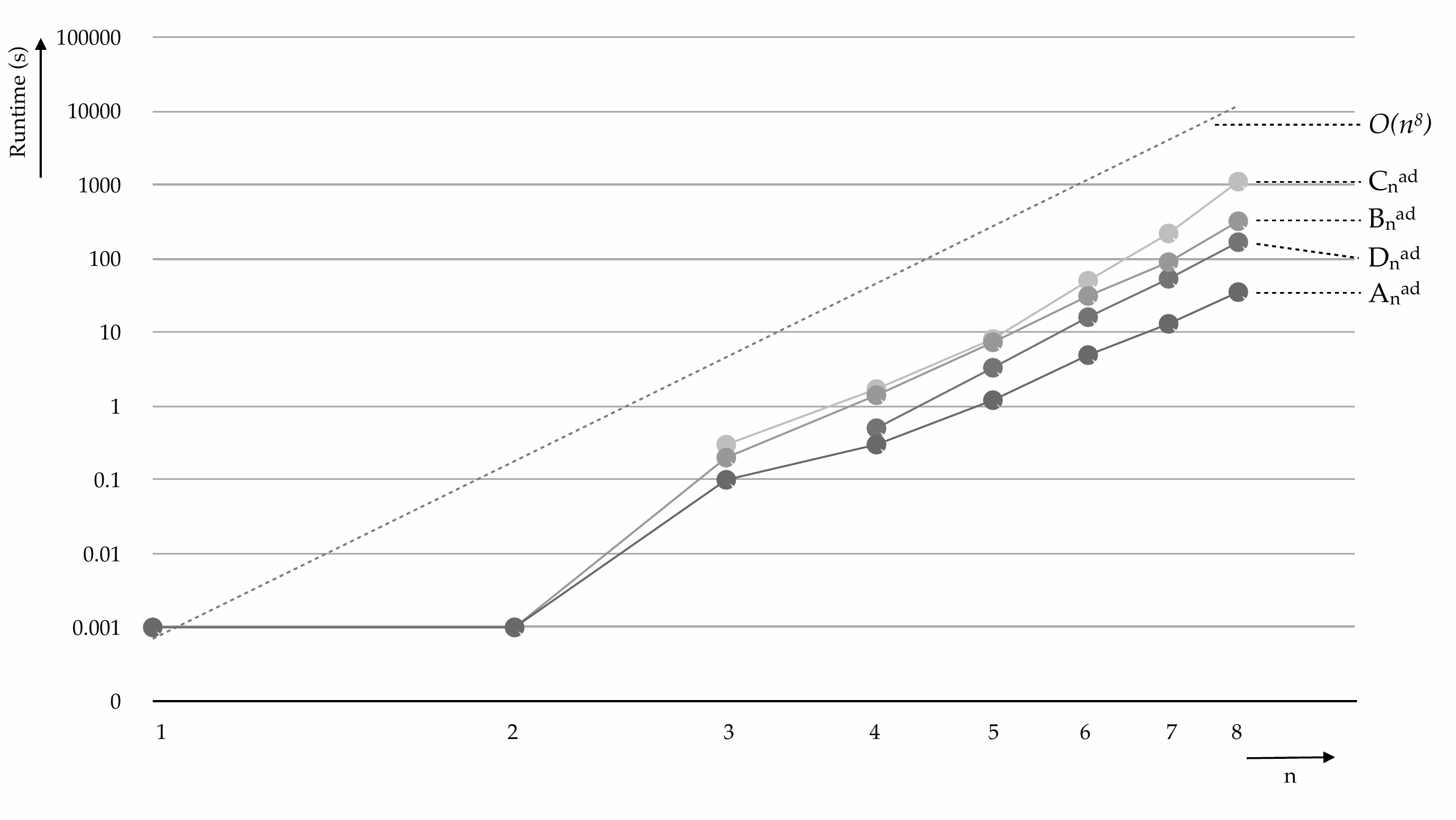}
	\caption{Runtimes for {\sc SplitMaximalToralSubalgebra2} for $\F = \GF(2^6)$}\label{fig_stsa_runtimes_byalg}
\end{figure}

\begin{figure}
	\includegraphics[width=145mm]{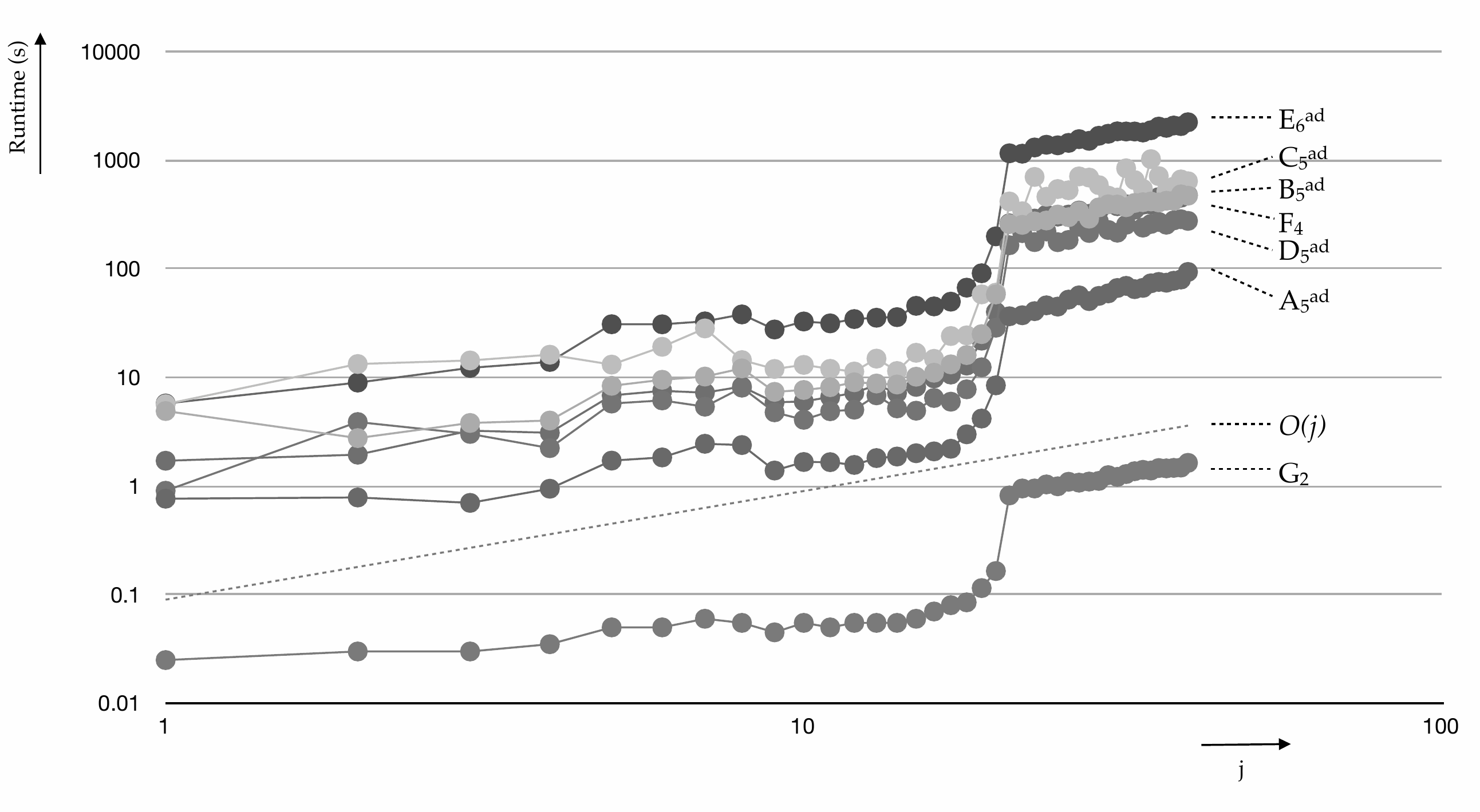}
	\caption{Runtimes for {\sc SplitMaximalToralSubalgebra2} for $\F = \GF(2^j)$, $1 \leq j \leq 40$}\label{fig_stsa_runtimes_byfld}
\end{figure}

\def\cprime{$'$}

\end{document}